\documentclass{amsart}
\usepackage{amsmath,amsthm,amssymb,enumerate}
\theoremstyle{plain}
\newtheorem{theorem}{Theorem}[section]
\newtheorem{lemma}[theorem]{Lemma}
\newtheorem{proposition}[theorem]{Proposition}
\newtheorem{corollary}[theorem]{Corollary}
\theoremstyle{definition}
\newtheorem{question}[theorem]{Question}
\newcommand{\Aut}{\mathrm{Aut}}
\newcommand{\Core}{\mathrm{Core}}
\newcommand{\supp}{\mathrm{supp}}
\newcommand{\N}{\mathbb{N}}
\newcommand{\Z}{\mathbb{Z}}
\newcommand{\GL}{\mathrm{GL}}
\newcommand{\Id}{\mathrm{Id}}
\begin{document}
\title[division subrings normalized by almost subnormal subgroups]{On division subrings normalized by almost subnormal subgroups in division rings}
\author[Trinh Thanh Deo]{Trinh Thanh Deo}

\author[Mai Hoang Bien]{Mai Hoang Bien}

\author[Bui Xuan Hai]{Bui Xuan Hai}\thanks{This work is funded by Vietnam National Foundation for Science and Technology Development (NAFOSTED) under Grant No. 101.04-2016.18.}
\email{ttdeo@hcmus.edu.vn; mhbien@hcmus.edu.vn, bxhai@hcmus.edu.vn}
\address{Faculty of Mathematics and Computer Science, VNUHCM - University of Science, 227 Nguyen Van Cu Str., Dist. 5, Ho Chi Minh City, Vietnam.} 

\keywords{division ring; almost subnormal subgroup; generalized group identity. \\
\protect \indent 2010 {\it Mathematics Subject Classification.} 16K20, 16K40, 16R50.}
\maketitle

\begin{abstract}
Let $D$ be a division ring with infinite center, $K$ a proper division subring of $D$ and $N$ an almost subnormal subgroup of the multiplicative group $D^*$ of $D$. The aim of this paper is to show that if $K$ is $N$-invariant and $N$ is non-central, then $K$ is central. Some examples of almost subnormal subgroups in division rings that are not subnormal are also given. 
\end{abstract}

\section{Introduction}
Let $G$ be a group. Recall that a subgroup $N$ of $G$ is said to be {\it almost subnormal} in $G$ if there exists a series of subgroups $$N=N_r\leq N_{r-1}\leq \ldots\leq N_1\leq N_0=G\eqno(1)$$ such that for any $0\le i<r$, $N_{i+1}$ is normal in $N_i$ or $[N_{i}:N_{i+1}]$ is finite. Such a series is called {\it an almost normal series} of $N$ in $G$ (see \cite{Pa_Ha_89}). We say that $N$ is an {\it almost $r$-subnormal subgroup} (or $N$ is an almost subnormal subgroup of {\it depth} $r$) if $r$ is the smallest number among all almost normal series $(1)$ of $N$ in $G$. By definition, every subnormal subgroup in a group is almost subnormal. The class of almost subnormal subgroups of skew linear groups was first mentioned in \cite{Wehr_93}. Recently, it was shown that if $D$ is a division ring with infinite center and if $n>1$, then every almost subnormal subgroup of $\GL_n(D)$ is normal \cite[Theorem 3.3]{Pa_NgBiHa_2016}. However, the case of degree $1$, that is, the group $D^*$, is totally different. 
The construction by Hazrat and Wadsworth in \cite[Example 8]{Haz-Wad} of non-normal maximal subgroups of finite index in certain division rings shows that there exist many division rings whose multiplicative groups contain almost $1$-subnormal subgroups which are not subnormal. In Section 2, we will show that for every positive integer $r$, there exists a division ring whose multiplicative group contains an almost $r$-subnormal subgroup which is not subnormal.

Let $D$ be a division ring with center $F$ and $K$ a division subring of $D$. Assume that $N$ is a non-central subgroup of $D^*$ such that $K$ is normalized by $N$, that is, $xKx^{-1}=K$ for any $x\in N$. Then, we say that $K$ is \textit{$N$-invariant}. The well-known Cartan-Brauer-Hua Theorem states that if $K^*$ is normal in $D^*$, that is, $K$ is $D^*$-invariant, then $K\subseteq F$ or $K=D$ \cite[p. 211, 13.17]{lam}. Herstein, Scott and Stuth then proved that if $N$ is a subnormal subgroup of $D^*$, then $K\subseteq F$ or $K=D$ (see \cite{Pa_HeSc_63}, \cite[p. 439]{scott} and \cite{Stuth}). This result is sometimes referred to as the Cartan-Brauer-Hua Theorem for subnormal subgroups in division rings. There is a vast number of applications on studying subgroups in division rings. 

In this note, we will extend these results for the case when $N$ is an almost subnormal subgroup of $D^*$ and $F$ is infinite.
We will prove that if $K$ is a division subring of $D$ such that $K$ is $N$-invariant, then $K\subseteq F$ or $K=D$. 

Our notation in this paper is standard. In particular, the symbol $[x,y]$ is denoted for $xyx^{-1}y^{-1}$.

\section{Examples of almost subnormal subgroups in division rings that are not subnormal}

In this section, for every positive integer $r$, we show the existence of a division ring $D$ containing an almost $r$-subnormal subgroups $N$ of $D^*$. Our example depends on Mal'cev-Neumann division rings of free groups over a base division ring, so we first recall some necessary results on this class of division rings. 

Let $G$ be a group. For an element $x\in G$, we construct a sequence of normal subgroups $$G=\langle x\rangle_0\vartriangleright \langle x\rangle_1 \vartriangleright \langle x\rangle_2\vartriangleright\ldots \vartriangleright\langle x\rangle_n\vartriangleright\ldots $$ of $G$ containing $x$ as follows. Let$\langle x\rangle_0=G$. For $n\geq 1$, assume that $\langle x\rangle_{n-1}$ is defined. Then, $\langle x\rangle_n$ is the smallest normal subgroup of $\langle x\rangle_{n-1}$ containing $x$. The subgroup $\langle x\rangle_n$ is called the $n$-th \textit{normal closure} (or the $n$-\textit{subnormal closure}) of $x$ in $G$. Recall that a subnormal subgroup $N$ of $G$ is called $r$-subnormal in $G$ if $$r=\min \{r\in\N\mid \text{ there exists a sequence } N=N_r\trianglelefteq N_{r-1} \trianglelefteq \ldots \trianglelefteq N_0=G\}.$$ It is easy to show that $\langle x\rangle_n$ is the intersection of all $s$-subnormal subgroups of $G$ containing $x$ with $s\le n$ (see \cite{Pa_Ol_17}). Before stating results on the subgroup $\langle x\rangle_n$, we borrow the following lemma.

\begin{lemma}[Greenberg]\label{Greenberg}
Let $G$ be a free subgroup and $H$ a finitely generated subgroup of $G$. If $H$ contains a non-trivial subnormal subgroup of $G$, then \mbox{$[G:H]<\infty$.}
\end{lemma}

\begin{proof} The lemma is implied by \cite{Pa_Gr_60} or see \cite[Theorem 2]{Pa_KaSo_69}.
\end{proof}

\begin{lemma}\label{infinite rank}
Let $G$ be a free group of infinite rank. For a non-identity element $x$ of $G$ and a natural number $n$, the $n$-th normal closure $\langle x\rangle_n$ of $x$ in $G$ is free of infinite rank. 
\end{lemma}
\begin{proof}
By the Nielsen-Schreier Theorem, every subgroup of a free group is free (see \cite[pp. 103-104]{Bo_St_93}), so $\langle x\rangle_n$ is free. By hypothesis, $\langle x\rangle_0=G$ is of infinite rank. Inductively, it suffices to show that $\langle x\rangle_1$ is of infinite rank. Assume that $\langle x\rangle_1$ is of finite rank. By Lemma~\ref{Greenberg}, $[G:\langle x\rangle_1]<\infty$. Let $G'$ be the commutator subgroup of $G$. Since $G$ is free of infinite rank, the abelian factor group $G/G'$ is non-torsion and infinitely generated. Let $H=G'.\langle x\rangle_1$ be the subgroup of $G$ generated by $G'$ and $\langle x\rangle_1$. Since $\langle x\rangle_1$ is of finite rank, $H/G'$ is finitely generated. It follows that $G/H$ is infinitely generated, which contradicts the fact that $[G:H]\le [G:\langle x\rangle_1]<\infty$. Thus, $\langle x\rangle_1$ is of infinite rank. The proof is complete.
\end{proof}

\begin{lemma}\label{maximal}
Let $G$ be a free group of infinite rank and $x$ an element of $G$. There exists a subgroup $H$ of $G$ satisfying the following conditions:
\begin{enumerate}[{\rm (1)}]
\item $H$ is maximal in $G$.
\item $H$ contains $x$.
\item $H$ is non-normal in $G$.
\item $[G:H]$ is finite.
\end{enumerate}
\end{lemma}

\begin{proof} Let $G$ be generated by $\{x_i\mid i\in I\}$, where $I$ is an infinite set, and assume that $x=x_{i_1}^{n_1}x_{i_2}^{n_2}\ldots x_{i_t}^{n_t}\in G$ for some $i_1, \ldots, i_t$ in $I$ and for some integers $n_1, \ldots, n_t.$ 
 Denote by $S_3$ the symmetric group on $\{1,2,3\}$. 
Since $I$ is infinite, there exist some indices $\lambda, \mu\in I\backslash\{i_1,i_2,\ldots,i_t\}$ with $\lambda\ne\mu$. Let $\phi : G\to S_3$ be the group morphism defined by $x_\lambda \mapsto (1~2),$ $x_\mu\mapsto (1~2~3)$ and $x_i\mapsto \Id$ for every $i\in I\backslash\{\lambda,\mu\}$. The group $N=\{\Id, (1~2)\}$ is non-normal maximal subgroup of $S_3$ of index $3$. We show that the subgroup $H=\phi^{-1}(N)$ of $G$ satisfies the lemma. Observe that $\phi$ is onto, $H$ is non-normal in $G$ and $[G:H]=3$ implying $H$ is maximal in $G$.
Finally, since $\phi(x)=\phi(x_{i_1})^{n_1}\phi(x_{i_2})^{n_2}\ldots \phi(x_{i_t})^{n_t}=\Id\in N$, one has $x\in H$.
\end{proof}

Next, we need to evaluate the depth of almost subnormal subgroups in free groups. In \cite{Pa_Ol_17}, A. Y. Olshanskii found a lower bound of depth for subnormal subgroups of $G$ contained in the $n$-th normal closure of a non-identity element. In fact, Olshanskii proved the following result. 

\begin{lemma}{\rm \cite[Corollary 1.3]{Pa_Ol_17}} \label{Ols_1} Let $G$ be a non-cyclic free group and $x$ be a non-identity element of $G$. Assume that $n$ is a natural number and $H$ is a non-trivial subgroup of $\langle x\rangle_n$. If $H$ is $s$-subnormal in $G$, then $s\ge n$.
\end{lemma}

By using this lemma, we will show the following.

\begin{lemma}\label{Ols_2}
Let $G$ be a free group of infinite rank, $x$ be a non-identity element of $G$ and $\langle x\rangle_n$ be the $n$-th normal closure of $x$ in $G$ for some positive integer $n$. If $\langle x\rangle_n$ contains a non-trivial subgroup $H$ which is almost $s$-subnormal in $G$, then $s\ge n$.
\end{lemma}

\begin{proof} Assume that $H$ is a non-trivial subgroup of $\langle x\rangle_n$ such that $H$ is almost $s$-subnormal in $G$. Let $$H=H_s <H_{s-1}<\ldots <H_1<H_0=G$$ be an almost normal series of $H$ in $G$, that is, for every $0\le i<s$, $H_{i+1}$ is normal in $H_i$ or $[H_{i}:H_{i+1}]$ is finite. We claim that $H$ contains a non-trivial subgroup $N$ which is $t$-subnormal in $G$ with $t\le s$. We first observe that in the case when $[H_{s-1}:H]<\infty$, we denote by $\Core(H)$ the core of $H$ in $H_{s-1}$. If $\Core(H)$ is trivial, then since $[H_{s-1}:\Core(H)]<\infty$, one has $H_{s-1}$ is finite, so is $H$. Moreover, by the Nielsen-Schreier Theorem (see \cite[pp. 103-104]{Bo_St_93}), $H$ is a free group. It implies that $H$ is trivial, a contradiction. Hence, $\Core(H_s)$ is non-trivial. Therefore, by replacing $H$ by $\Core(H)$ if necessary, without loss of generality, we assume that $H=H_s$ is normal in $H_{s-1}$. We show the claim by induction on $s$. If $s=1$, then $H$ is normal in $G$, and the claim holds trivially. Assume that the claim is true for $s-1$, that is, every non-trivial almost $(s-1)$-subnormal subgroup of $G$ contains a non-trivial group which is $(t-1)$-subnormal in $G$ with $t\le s$. In particular, $H_{s-1}$ contains a non-trivial subgroup $N_{s-1}$ which is $(t-1)$-subnormal in $G$. Put $N_s=H_s\cap N_{s-1}$. It is obvious that $N_s$ is normal in $N_{s-1}$. 
We consider two cases:
\bigskip

\noindent\emph{Case 1: $H_s\subseteq N_{s-1}$ or $N_{s-1}\subseteq H_s$.} 
\bigskip

It implies obviously that $N_s=H_s\cap N_{s-1}$ is non-trivial. 
\bigskip

\noindent\emph{Case 2: $H_s\backslash N_{s-1}\ne \emptyset$ and $N_{s-1}\backslash H_s\ne \emptyset$.}
\bigskip

 Let $h\in H_s\backslash N_{s-1}$ and $k\in N_{s-1}\backslash H_s$. Then, $h,k\in H_{s-1}$, so $$(hkh^{-1})k^{-1}=h(kh^{-1}k^{-1})\in H_s\cap N_{s-1}=N_s.$$
If $hkh^{-1}k^{-1}=1$, then the subgroup $\langle h,k\rangle$ of $G$ is abelian. Again by the Nielsen-Schreier Theorem, $\langle h,k\rangle$ is a free group, which implies that $\langle h,k\rangle\cong \Z$. As a corollary, $h=k^\alpha$ for some $\alpha\in \Z$ or $k=h^\beta$ for some $\beta\in \Z$. This contradicts the assumption $h\in H_s\backslash N_{s-1}$ and $k\in N_{s-1}\backslash H_s$. Therefore, $N_s$ is non-trivial.

Both cases lead us to the conclusion that $N_s$ is non-trivial. It means that $N_s$, being a subgroup of $H_s$, is $t$-subnormal in $G$ with $t\le s$. The claim is shown. Now, since $N_s\subseteq H_s\subseteq \langle x\rangle_n$, by Lemma~\ref{Ols_1}, $t\ge n$. Thus, $s\ge t\ge n$. The proof is now complete.
\end{proof}

Now, we fix some notation. Assume that $G$ is a free group generated by a countable set of indeterminates $\{x_i\mid i\in \N\backslash \{0\} \}$
 with the \textit{Magnus order} (see \cite[Theorem 6.31]{lam} or see \cite[Section 5.7]{Bo_MaKaSo_76} for further details). This order sometimes is called \textit{the dictionary order} of free groups since it is induced from the dictionary order on $\{x_i\mid i\in \N\backslash \{0\} \}$. By reordering, we assume that  
 $x_1<x_2<\ldots$ Observe that the order is a total order, that is, for every $a,b,c\in G$, the following conditions hold: 
(1) $a\le b$ or $b\le a$ (connex property); 
(2) if $a\le b$ and $b\le a$, then $a=b$ (antisymmetry); and 
(3) if $a\le b$ and $b\le c$, then $a\le c$ (transitivity). 
Moreover, if $a\le b$, then $ac\le bc$ and $ca\le cb$. Hence, $G$ is a total ordered group with dictionary order. Recall that a subset $S$ of $G$ is called \textit{well-ordered} (briefly, \textit{WO}) if every non-empty subset of $S$ has the least element. For a non-empty WO set $S$ of $G$, we denote by $\min(S)$ the least element of $S$. 

Next, let $K$ be a division ring. From the general Mal'cev-Neumann construction of a Laurent series ring (we refer to \cite[(14.21), p. 231]{lam} for details), we consider the Mal'cev-Neumann division ring $K((G,\Phi))$ of $G$ over $K$ with respect to a group homomorphism $\Phi: G\to \Aut (K), g\mapsto \Phi_g$. 

More specifically, the Mal'cev-Neumann division ring $K((G,\Phi))$ consists of formal (but not necessarily finite) sums of the form $\alpha=\sum\limits_{g\in G}a_g g,$ where $a_g\in K$ and $\supp(\alpha)=\{g\in G\mid a_g\ne 0\}$ is WO.
In $K((G,\Phi))$, addition and multiplication are defined for $\alpha=\sum\limits_{g\in G}a_gg$ and $\beta=\sum\limits_{g\in G}b_gg$ by: $$\alpha+\beta=\sum\limits_{g\in G}(a_g+b_g)g$$ and $$\alpha\beta=\sum\limits_{t\in G}\Bigl(\sum\limits_{gh=t} a_g\Phi_g(b_h)\Bigr)t.$$ 

Now, we present the main result of this section by showing that the multiplicative group $K((G,\Phi))^*$ contains an almost subnormal subgroup of arbitrary depth $r$ which is not subnormal.

\begin{theorem}\label{exa}
Let $G$ be a free group generated by a countable set of indeterminates $\{x_i\}_{i\ge 1}$ with dictionary order, $K$ be a division ring and $\Phi : G\to \Aut(K)$ be a group morphism. Then, for every positive integer $r$, the multiplicative group $D^*$ of the Mal'cev-Neumann division ring $D=K((G,\Phi))$ contains an almost $r$-subnormal subgroup $N$ which is not subnormal in $D^*$.
\end{theorem}

\begin{proof} Let $x\in G\backslash\{1\}$ and consider the sequence of normal subgroups $$G=\langle x\rangle_0\vartriangleright \langle x\rangle_1\vartriangleright\ldots \vartriangleright\langle x\rangle_{r-1}.$$ By lemmas~\ref{infinite rank} and \ref{maximal}, let $H$ be a maximal subgroup of $\langle x\rangle_{r-1}$ which is non-normal, of finite index in $\langle x\rangle_{r-1}$ and contains $x$. Then, in particular, $H$ is almost subnormal in $G$. Now put $$d : D^*\to G, \alpha\mapsto \min(\supp(\alpha)).$$ We claim that $$\min(\supp(\alpha.\beta))=\min(\supp(\alpha)).\min(\supp(\beta))\text{ for every }\alpha,\beta\in D.$$ Let $\alpha=\sum\limits_{g\in G}a_g g, \beta=\sum\limits_{g\in G}b_g g\in D$. Put 
$$A=\supp(\alpha),\; u=\min (A),\; B=\supp(\beta) \text{ and }v=\min (B).$$
Then, $\alpha=\sum\limits_{g\in A}a_g g$ and $\beta=\sum\limits_{h\in B}b_h h,$ where $a_g\ne 0$ and $b_h\ne 0$ for every $g\in A$ and $h\in B$. Observe that $u\le g$ and $v\le h$ for every $g\in A$ and $h\in B$, so $uv\le t$ for every $t\in A.B$. Since $\alpha\beta = \sum\limits_{t\in G}c_tt$, where $c_t=\sum\limits_{gh=t, g\in A, h\in B} a_g\Phi_g(b_h)$, we have $uv\le \min(\supp(\alpha\beta))$. Moreover, $c_{uv}=\sum\limits_{gh=uv, g\in A, h\in B} a_g\Phi_g(b_h)$. One has $c_{uv}=a_u\Phi_u(b_v)\ne 0$ because $uv\le gh$ for every $g\in A$ and $h\in B$. So, $uv\in \supp(\alpha\beta)$. Therefore, $uv=\min(\supp(\alpha\beta))$, and the claim is shown. Consequently, $d$ is a group morphism and obviously $d$ is surjective.
Put $$N=d^{-1}(H)=\{\alpha\in D\mid \min(\supp(\alpha))\in H \}.$$ We will show that $N$ is an almost $r$-subnormal subgroups of $D^*$ and it is not subnormal in $D^*$. 

First, we show that $N$ is almost $r$-subnormal in $D^*$. Put $$N_i=d^{-1}(\langle x\rangle_i)=\{\alpha\in D\mid \min(\supp(\alpha))\in \langle x\rangle_i \}$$ for every $0\le i\le r-1$. Then, $N<N_{r-1}\triangleleft N_{r-2}\triangleleft\ldots \triangleleft N_0=D^*$ with $[N_{r-1}:N]=[\langle x\rangle_{r-1}: H]<\infty$ since $d$ is surjective. Therefore, $N$ is almost subnormal in $D^*$. Moreover, if $$N=X_\ell<X_{\ell-1}<\ldots<X_{1}<X_0=D^*$$ is an almost normal series of $N$ in $D^*$, then $$H=d(X_\ell)\le d(X_{\ell-1})\le \ldots\le d(X_{1})\le d(X_0)=G$$ is also an almost normal series of $H$ in $G$ since $d$ is surjective. According to Lemma~\ref{Ols_2}, $\ell\ge r$, which implies that $N$ is almost $r$-subnormal in $D^*$. 

Next, we claim that $N$ is not subnormal in $D^*$. Assume by contrary, $N$ is subnormal in $D^*$, and let $$N=N_\ell \trianglelefteq N_{\ell-1}\trianglelefteq\ldots\trianglelefteq N_1\trianglelefteq N_0=D^*$$ be a normal series of $N$ in $D^*$. If $\ell<r$, then we add $N_{\ell+1}=N_{\ell+2}=\cdots=N_r$ to the normal series, so without loss of generality, we assume that $\ell\ge r$. Put $M_i=d(N_i)$ for every $0\le i\le \ell$. Then, since $d$ is surjective, $$H=M_\ell \trianglelefteq M_{\ell-1}\trianglelefteq\ldots\trianglelefteq M_1\trianglelefteq M_0=G.$$
By definition, $\langle x\rangle_i$ is contained in $M_i$ for every $0\le i\le \ell$. Therefore, if $L_{i}=M_i\cap \langle x\rangle_i$ for every $1\le i\le r$, then $$H=M_\ell\cap \langle x\rangle_r\trianglelefteq M_{\ell -1}\cap \langle x\rangle_r\trianglelefteq \ldots \triangleleft M_r\cap \langle x\rangle_r=\langle x\rangle_r.$$
It implies that $H$ is subnormal in $\langle x\rangle_r$ which contradicts the fact that $H$ is maximal non-normal in $\langle x\rangle_r$. Thus, $N$ is non-subnormal in $D^*$, and the proof is now complete.
\end{proof}

Using same arguments in the proof of Lemma~\ref{Ols_2}, we can show that every non-trivial almost subnormal subgroup of a free group $G$ of infinite rank contains a non-trivial subnormal subgroup of $G$ and by repeating arguments in Theorem~\ref{exa}, we may show that if $D$ is the Mal'cev-Neumann division ring $K((G,\Phi))$ of a free group $G$ of infinite rank and $N$ is a non-central almost subnormal subgroup of $D^*$, then $N$ contains a non-central subnormal subgroup of $D^*$. It is natural to propose the following question.
\begin{question}Let $D$ be a division ring and $N$ almost subnormal subgroup of $D^*$. If $N$ is non-central, then is it true that $N$ contains a non-central subnormal subgroup of $D^*$?
\end{question}

\section{$N$-invariant division subrings}
The technique we use in this section is generalized rational identities.  Now, let us recall the definition of generalized rational identities. Let $D$ be a division ring with center $F$ and $X=\{x_1,\ldots,x_m\}$ be a set of $m$ (noncommuting) indeterminates. We denote by $F\langle X\rangle$ the free $F$-algebra on $X$, by $D\langle X\rangle$ the free product of $D$ and $F\langle X\rangle$ over $F$, and by $D(X)$ the universal division ring of fractions of $D\langle X\rangle$. The existence of $D(X)$ was shown and studied deeply in \cite[Chapter 7]{Bo_Cohn}. One calls an element $f(X)\in D(X)$ a \textit{generalized rational polynomial}. If $f(X)\in D(X)$, then by \cite[Theorem 7.1.2]{Bo_Cohn}, $f(X)$ is an entry of the matrix $A^{-1}$, where $A\in M_n(D\langle X\rangle)$ for some positive integer $n$ such that $A$ is invertible in $M_n(D(X))$.
Let $c=(c_1,\ldots,c_m)\in D^m$ and $\alpha_c\colon D\langle X\rangle \to D$ be the ring homomorphism defined by $\alpha(x_i)=c_i$. For any $n\in \N, $ let $S(c,n)$ be the set of all square matrices $(f_{ij}(X))$ of degree $n$ over $D\langle X\rangle$ such that the matrix $(f_{ij}(c))$ is invertible in $M_n(D)$. Let $S(c)=\bigcup\limits_{n\ge 1}S(c,n)$ and $E(c)$ be the subset of $D(X)$ consisting of all entries of $A^{-1}$, where $A$ ranges over $S(c)$. Then, $E(c)$ is a subring of $D(X)$ containing $D\langle X\rangle$ as a subring. Moreover, there is a ring homomorphism $\beta_c\colon E(c)\to D$ which extends $\alpha_c$ and every element of $E(c)$ is invertible if and only if the matrix mapped by $\beta_c$ is not zero. Let $f(X)\in D(X)$ and $c\in D^m$. If $f(X)\in E(c)$, then we say that $f(X)$ is defined at $c$ and $\beta_c(c)$ is denoted by $f(c)$. For any $f(X)\in D(X)$, the set of all $c\in D^m$ such that $f(X)$ is defined at $c$ is called the {\it domain} of $f(X)$ and is denoted by $\mathrm{Dom}_D(f)$. Let $S\subseteq D$ and $f(X)$ be a non-zero element in $D(X)$. If $f(c)=0$ for all $c\in S^m \cap \mathrm{Dom}_D(f)$, then we say that $f=0$ is a {\it generalized rational identity} of $S$ or $S$ {\it satisfies the generalized rational identity} $f=0$. In this paper, we borrow the following result.
\begin{lemma}\cite[Theorem 1.1]{Pa_HaBiDu_16}\label{hdb}
Let $D$ be a division ring with infinite center. If $D^*$ contains a non-central almost subnormal subgroup satisfying a generalized rational identity, then $D$ is centrally finite, that is, $D$ is finite-dimensional over its center.
\end{lemma}

In this section, we mainly work with a special class of generalized rational identities, namely, generalized group identities. An element $$w(x_1,x_2,\ldots,x_n)=a_1x_{i_1}^{m_1}a_2x_{i_2}^{m_2}\ldots a_tx_{i_t}^{m_t}a_{t+1}\in D(X)$$ is called a \textit{generalized group monomial} over $D^*$ if $a_1,a_2,\ldots,a_{t+1}\in D^*$. Additionally, assume that $w$ is non-identity. Let $H$ be a subgroup of $D^*$.
We say that $w=1$ is a \emph{generalized group identity} of $H$ or $H$ \emph{satisfies the generalized group identity} $w=1$ if $w(c_1,c_2,\ldots,c_n)=1$ for every $c_1,c_2,\ldots,c_n\in H$. For results on generalized group identities in division rings, we prefer \cite{Pa_GoMi_82} and \cite{Pa_To_85}. Results on generalized group identities of almost subnormal subgroups in division rings we use in this paper are from \cite{Pa_NgBiHa_2016}. In fact, we will use the following result.

\begin{proposition}{\rm \cite[Theorem 2.2]{Pa_NgBiHa_2016}}\label{pro:3}
Let $D$ be a division ring with infinite center $F$ and assume that $N$ is an almost subnormal subgroup of $D^*$. If $N$ satisfies a generalized group identity over $D^*$, then $N$ is central, that is, $N\subseteq F$.
\end{proposition}
Assume that $H$ is a subgroup of finite index $n$ in a group $G$. If $H$ is normal in $G$, then $g^n\in H$ for any element $g\in G$. If $H$ is non-normal in $G$, then by Poincare's theorem, it is easy to see that 
$g^{n!}\in H$ for any element $g\in G$.

Let $a\in D^*$. Assume that $N$ is a non-central almost subnormal subgroup of $D^*$ with an almost normal series $$N=N_r\leq N_{r-1}\leq \ldots\leq N_1\leq N_0=D^*.$$
We construct a series of subgroups $H_n$ of $D^*$ depending on $N$ and $a$ as follows:
 
Put $H_0=D^*$. For any integer $n>0$, if $N_n$ is normal in $N_{n-1}$, then we put $H_n:=\langle bab^{-1}\mid b\in H_{n-1}\rangle$. Otherwise, $[N_{n-1}:N_n]=\ell_n<\infty$, we put $H_n:=\langle b^{\ell_n!}\mid b\in H_{n-1}\rangle$. Hence, we get the following new sequence of subgroups 
$$H=H_r\leq H_{r-1}\leq \ldots \leq H_1\leq H_0=D^*.$$ Moreover, we have the following lemma.

\begin{lemma} \label{lem:4} Let $N_1,N_2,\ldots,N_r, H_1,H_2,\ldots,H_r$ and $a$ be as above. Then, the following statements hold:
\begin{enumerate}[{\rm (1)}]
\item For every $0\le n<r$, $H_n\trianglelefteq H_{n-1}$ in case $N_n\trianglelefteq N_{n-1}$, and $b^{\ell_n!}\in H_n$ for every $b\in H_{n-1}$ in case $[N_{n}:N_{n-1}]=\ell_n<\infty$.
\item  If $a\in N$, then $H_n\leq N_n$ for any $0\le n\le r$.
\item For any $0\le n\le r$, if $c\in D^*$ such that $ac=ca$, then $cH_nc^{-1}\leq H_n.$
\end{enumerate}
\end{lemma}
\begin{proof} (1) It follows immediately from the definition of $H_n$.
\bigskip

(2) We show this assertion by induction on $n$. It is trivial that $H_0\subseteq N_0$. Assume that $H_{n-1}\le N_{n-1}$. We will show $H_n\le N_n$. There are two cases.
\bigskip

\noindent{\it Case 1: $N_n\trianglelefteq N_{n-1}$}. 
\bigskip

Since $a\in N\le N_n$, one has $bab^{-1}\in N_n$ for every $b\in N_{n-1}$. In particular, $bab^{-1}\in N_n$ for every $b\in H_{n-1}$. Hence, $H_n:=\langle bab^{-1}\mid b\in H_{n-1}\rangle\le N_{n}$.
\bigskip

\noindent{\it Case 2. $[N_{n-1}:N_n]=\ell_n<\infty$}. 
\bigskip

Then, $b^{\ell_n!}\in N_n$ for every $b\in N_{n-1}$. In particular, $b^{\ell_n!}\in N_n$ for every $b\in H_{n-1}$. Therefore, $$H_n:=\langle b^{\ell_n!}\mid b\in H_{n-1}\rangle\le N_n.$$

In both cases, we have $H_n\le N_n$, so (2) holds. 
\bigskip

(3) We prove this statement by induction on $n$. It is trivial that $cH_0c^{-1}=H_0$. Assume that $0<n\leq r$ and $cH_{n-1}c^{-1}\leq H_{n-1}$. We must show $cH_nc^{-1}\leq H_n$. If $N_n$ is normal in $N_{n-1}$, then $H_n=\langle bab^{-1}\mid b\in H_{n-1}\rangle$. Hence,
$$cH_nc^{-1}=\langle c(bab^{-1})c^{-1}\mid b\in H_{n-1}\rangle =\langle (cbc^{-1})a(cbc^{-1})^{-1}\mid b\in H_{n-1} \rangle\leq H_n.$$
In the case when $[N_{n-1}:N_n]=\ell_n<\infty$, then $b^{\ell_n!}\in N_n$ for any $b\in N_{n-1}$. By definition, $H_n=\langle b^{\ell_n!}\mid b\in H_{n-1}\rangle$, so $$cH_nc^{-1}=\langle cb^{\ell_n!}c^{-1}\mid b\in H_{n-1}\rangle=\langle (cbc^{-1})^{\ell_n!}\mid b\in H_{n-1}\rangle\leq H_n.$$ 

The proof is now complete.
\end{proof}

Now, keeping the assumption and the notation for $N$ as above, we build inductively a sequence of non-identity elements $w_n$ in $D(x,y)$ depending on $N$ as follows.
Put $w_0(x,y)=x$. For any integer $n\geq 1,$ put
$$w_n(x,y):=\begin{cases}[w_{n-1}(x,y),y],&\text{if }N_{n} \text{ is normal in }N_{n-1}\text{ or } n>r\\
(w_{n-1}(x,y))^{\ell_n!},&\text{if } [N_{n-1}:N_n]=\ell_n<\infty.
\end{cases}$$

We claim that if $a\in D^*$ and $b\in N$, then $w_n(a,b)\in N_n$ for any $n\le r$ and $w_n(a,b)\in N_r=N$ for any $n> r$. We will first show the claim by induction on $0\le n\le r$. It is clear that $w_0(a,b)=a\in N_0=D^*$. Assume that $w_{n-1}(a,b)\in N_{n-1}$. If $N_n$ is normal in $N_{n-1}$, then $$w_n(a,b)=w_{n-1}(a,b)b(w_{n-1}(a,b))^{-1}b^{-1}=(w_{n-1}(a,b)b(w_{n-1}(a,b))^{-1})b^{-1}\in N_n.$$
If $[N_{n-1}:N_n]=\ell_n<\infty$, then $c^{\ell_n!}\in N_n$ for every $c\in N_{n-1}$, so $w_n(a,b)=(w_{n-1}(a,b))^{\ell_n!}\in N_n$. Thus, we showed that $w_n(a,b)\in N_n$ for any $n\le r$. As a corollary, $w_n(a,b)\in N_r=N$ for any $n> r$.

Now, assume that $K$ is a division subring of $D$ such that $K$ is $N$-invariant and $K\not\subseteq F$. Suppose that $h\in (K\cap N)\backslash F$ and $g\in N\backslash \{-1\}$, and consider the elements $$u_n(h,g):=w_n((1+g)h(1+g)^{-1}, g),$$ 
$$v_n(h,g):=w_n((1+g)^{-1}h(1+g), g).$$ 
 
\begin{lemma}\label{lem:5} For any $n \ge 0$,
$u_n(h,g)=(1+g)\phi_n(h,g)(1+g)^{-1}$ and $v_n(h,g)=(1+g)^{-1}\phi_n(h,g)(1+g),$ where $\phi_n(h,g)$ is an element of the subgroup $\langle h,g\rangle$ of $K^*\cap N$.
\end{lemma}
\begin{proof} We prove the lemma by induction on $n$. We have $$u_0(h,g)=w_0((1+g)h(1+g)^{-1}, g)=(1+g)h(1+g)^{-1}$$ and $$v_0(h,g)=w_0((1+g)^{-1}h(1+g), g)=(1+g)^{-1}h(1+g).$$ Hence, $\phi_0(h,g)=h\in \langle h,g\rangle$. 

Assume that $u_{n-1}=w_{n-1}((1+g)h(1+g)^{-1},g)=(1+g)\phi_{n-1}(h,g)(1+g)^{-1}$ and $v_{n-1}=w_{n-1}((1+g)^{-1}h(1+g),g)=(1+g)^{-1}\phi_{n-1}(h,g)(1+g)$ where $\phi_{n-1}(h,g)$ is an element of the subgroup $\langle h,g\rangle$. Now consider $u_{n}=w_{n}((1+g)h(1+g)^{-1},g)$ and $v_{n}=w_{n}((1+g)^{-1}h(1+g),g)$ with two cases:
\bigskip

\noindent{\it Case 1: $N_n$ is normal in $N_{n-1}$.}
\bigskip

Then, 
$$\begin{aligned}
u_n(h,g)&=w_{n}((1+g)h(1+g)^{-1},g)\\
&=w_{n-1}((1+g)h(1+g)^{-1},g)g (w_{n-1}((1+g)h(1+g)^{-1},g))^{-1}g^{-1}\\
&= (1+g)\phi_{n-1}(h,g)(1+g)^{-1} g (1+g)(\phi_{n-1}(h,g))^{-1}(1+g)^{-1}\\
&=(1+g)\phi_{n-1}(h,g)g(\phi_{n-1}(h,g))^{-1}(1+g)^{-1}.
\end{aligned}$$
It is similar that $$v_n(h,g)=(1+g)^{-1}\phi_{n-1}(h,g)g(\phi_{n-1}(h,g))^{-1}(1+g).$$
It implies that $\phi_n(h,g)=\phi_{n-1}(h,g)g(\phi_{n-1}(h,g))^{-1}\in \langle h,g\rangle$.
\bigskip

\noindent{\it Case 2: $[N_{n-1}: N_n]=\ell_n<\infty$.}
\bigskip

Then, 
$$u_n(h,g)=w_{n}((1+g)h(1+g)^{-1},g)^{\ell_n!}=(1+g)(\phi_{n-1}(h,g))^{\ell_n!}(1+g)^{-1}.$$
Similarly, $v_n(h,g)=(1+g)^{-1}(\phi_{n-1}(h,g))^{\ell_n!}(1+g)$. Hence, 
$$\phi_n(h,g)=(\phi_{n-1}(h,g))^{\ell_n!}\in \langle h,g\rangle.$$

The proof is now complete.
\end{proof}

\begin{lemma}\label{lem:6}
Let $D$ be a division ring with infinite center $F$, and assume that $N$ is a non-central almost subnormal subgroup of $D^*$. If $K$ is a non-central $N$-invariant division subring of $D$, then $K\cap N$ is non-abelian. In particular, $K\cap N\not \subseteq F$.
\end{lemma}
\begin{proof} 
We will show that $K^*\cap N$ is not abelian. Assume that $K^*\cap N$ is abelian, and take any $a\in K\backslash F$. For any $b\in N$, we claim that $w_{r}(a,b)\in K^*\cap N$. As we have noted above, $w_{r}(a,b)\in N_{r}=N$. Now, we prove by induction that $w_n(a,b)\in K^*$ on $n\geq 0$. In the beginning, we have $w_0(a,b)=a\in K^*$. Suppose that $w_{n-1}(a,b)\in K^*$, then consider $w_{n}(a,b)$. If $N_{n}$ is normal in $N_{n-1}$, then $w_{n}(a,b)=[w_{n-1}(a,b),b]=w_{n-1}(a,b)b(w_{n-1}(a,b))^{-1}b^{-1}$. It implies that $w_{n+1}(a,b)\in K^*$ because $b(w_{n-1}(a,b))^{-1}b^{-1}\in K^*$ in view of the assumption that $K$ is $N$-invariant. If $[N_{n-1}:N_n]=\ell_n<\infty$, then $w_{n}(a,b)=(w_{n-1}(a,b))^{\ell_n!}\in K^*$. By the inductive assumption, $w_n(a,b)\in K^*$ for any $n\geq 0$. In particular, $w_r(a,b)\in K^*\cap N$, and the claim is proved. Thus, $w_r(a,b)w_r(c,d)(w_r(a,b))^{-1}(w_r(c,d))^{-1}=1$ for any $a,c\in K\backslash F$ and $b,d\in N$ because $K^*\cap N$ is abelian. Since $a,c\not\in F$, $w_r(a,x)w_r(c,y)(w_r(a,x))^{-1}(w_r(c,y))^{-1}$ is non-identity in $D(x,y)$. Hence, $N$ satisfies the generalized group identity $$w_r(a,x)w_r(c,y)(w_r(a,x))^{-1}(w_r(c,y))^{-1}=1.$$ By Proposition~\ref{pro:3}, $N$ is central, a contradiction.
\end{proof}

The following lemma is a special case of Theorem \ref{th:1} when $K$ is assumed to be a subfield of $D$. 

\begin{lemma}\label{lem:7}
Let $D$ be a division ring with infinite center $F$, and assume that $N$ is a non-central almost subnormal subgroup of $D^*$. If $K$ is an $N$-invariant subfield of $D$, then $K\subseteq F$. 
\end{lemma}
\begin{proof} Assume that $K$ is a subfield of $D$ which is non-central and $N$-invariant. Let $a\in K\backslash F$. Then, for every $b\in N,$ one has $bab^{-1}\in K$, so $bab^{-1}a ba^{-1}b^{-1}a^{-1}=(bab^{-1})a (bab^{-1})^{-1}a^{-1}=1$. Moreover, $xax^{-1}a xa^{-1}x^{-1}a^{-1}$ is non-identity in $D(x)$ as $a\not\in F$. Hence, $N$ satisfies the generalized group identity $$xax^{-1}a xa^{-1}x^{-1}a^{-1}=1.$$
Now, by Proposition~\ref{pro:3}, $N$ is central and this contradicts the hypothesis. The proof is now complete.
 \end{proof}

The following lemma is essential in the proof of Theorem~\ref{th:1}.
\begin{lemma}\label{lem:9} Let $D$ be a division ring with infinite center $F$, and assume that $N$ is a non-central almost subnormal subgroup of $D^*$. If $K$ is a non-central $N$-invariant division subring of $D$, then $C_D(a)\subseteq K$ for any $a\in (K\cap N)\backslash F$.
\end{lemma}
\begin{proof}
Since $K\not\subseteq F$ by Lemma~\ref{lem:6}, we have $K\cap N\not\subseteq F$. Assume that there exists $a\in (K\cap N)\backslash F$ such that $C_D(a)\not\subseteq K$. Let $$N=N_r\leq N_{r-1}\leq \ldots\leq N_1\leq N_0=D^*$$ be an almost normal series of $N$ in $D^*$.

Consider the subgroups $H_n$ constructed as in Lemma~\ref{lem:4} for $0\le n\le r$, that is, $H_0=D^*$. For any integer $n>0$, if $N_n$ is normal in $N_{n-1}$, then we put $H_n:=\langle bab^{-1}\mid b\in H_{n-1}\rangle$. Otherwise, that is, $[N_{n-1}:N_n]=\ell_n<\infty$, we put $H_n:=\langle b^{\ell_n!}\mid b\in H_{n-1}\rangle$.
 Let $g\in H_r$ be arbitrary. Since $H_r\leq N_r=N$, $gag^{-1}\in K^*$. We first claim that $$C_D(a)=C_D(gag^{-1})=gC_D(a)g^{-1}.$$ For $b\in C_D(a)$, put $$h:=b(gag^{-1})b^{-1}=(bgb^{-1})a(bgb^{-1})^{-1}.$$ By Lemma~\ref{lem:4}, $bgb^{-1}\in bH_rb^{-1}\leq H_r\leq N_r\leq N$. Since $K$ is $N$-invariant, $h\in K$ and $b(gag^{-1})=hb$. Similarly, since $b+1\in C_D(a)$, $(b+1)(gag^{-1})=h'(b+1)$ for some $h'\in K$. Hence, $$gag^{-1}=(b+1)(gag^{-1})-b(gag^{-1})=(h'-h)b+h'.$$ Recall that $gag^{-1}\in K$, so, if $h'\ne h$ then $$b=(h'-h)^{-1}(gag^{-1}-h')\in K.$$ Therefore, if $b\in C_D(a)\backslash K$, then $h=h'$, equivalently, $$b(gag^{-1})=hb=h'b=(gag^{-1})b.$$ 
Hence, $b\in C_D(gag^{-1})$ for $b\in C_D(a)\backslash K$. As a result, $C_D(a)\backslash K\subseteq C_D(gag^{-1})$. Now, for any $b'\in K\cap C_D(a)$, one has $b+b'\in C_D(a)\backslash K\subseteq C_D(gag^{-1})$, so $(b+b')(gag^{-1})=(gag^{-1})(b+b')$, equivalently, $b'(gag^{-1})=(gag^{-1})b'$. Thus, $$C_D(a)\cap K\subseteq C_D(gag^{-1}).$$ As a corollary, $C_D(a)\subseteq C_D(gag^{-1})$. Replacing $a$ by $gag^{-1}\in K\cap N$ and $g$ by $g^{-1}$, we have $C_D(gag^{-1})\subseteq C_D(g^{-1}(gag^{-1})g)=C_D(a)$. Hence, $C_D(a)=C_D(gag^{-1})$. 
Now, we prove the equality $C_D(a)=gC_D(a)g^{-1}$. For any $b\in C_D(a)=C_D(gag^{-1})$, one has $$(gbg^{-1})a=gb(g^{-1}ag)g^{-1}=g(g^{-1}ag)bg^{-1}=a(gbg^{-1}),$$ so $gbg^{-1}\in C_D(a)$. 
Therefore, the inclusion $gC_D(a)g^{-1}\subseteq C_D(a)$ holds for an arbitrary element $g\in H_r$. So, it follows $g^{-1}C_D(a)g \subseteq C_D(a)$, and consequently we have $C_D(a)=gC_D(a)g^{-1}$.
The claim is now proved.

Now, we claim that the center $Z(C_D(a))$ of $C_D(a)$ is $H_r$-invariant. For any $g\in H_r, c\in Z(C_D(a))$ and $d\in C_D(a)$, we have $(gcg^{-1})d=gc(g^{-1}dg)g^{-1}=g(g^{-1}dg)cg^{-1}=d(gcg^{-1})$. By applying Lemma~\ref{lem:7} for the subgroup $H_r$ and subfield $Z(C_D(a))$, we have $Z(C_D(a))\subseteq F$. In particular, $a\in F$, a contradiction. Thus, the proof is now complete. 
\end{proof}

\begin{lemma}\label{lem:8} {\rm \cite[14.3.6, p. 435]{scott}}
Let $K$ be a division subring of $D$. If $g\in D\backslash K$ such that $gKg^{-1}\subseteq K$, then $K\cap (1+g)K(1+g)^{-1}=C_D(g)\cap K$.
\end{lemma}

\begin{corollary}\label{cor:10}
Let $D$ be a division ring with infinite center $F$, and assume that $N$ is a non-central almost subnormal subgroup of $D^*$. Assume that $K$ is an $N$-invariant division subring of $D$. If there exists $g\in N\backslash K$, then $$N\cap K\cap (1+g)K(1+g)^{-1}\subseteq F.$$
\end{corollary}
\begin{proof}
Assume that $N\cap K\cap (1+g)K(1+g)^{-1}\subseteq F$ for every $g\in N\backslash K$. By Lemma~\ref{lem:8}, there exists $h\in (N\cap K\cap C_D(g))\backslash F$. Hence, $g\in C_D(h)$, but by Lemma~\ref{lem:9}, $C_D(h)\subseteq K$, so $g\in K$ which contradicts the assumption.
\end{proof}

Now, we are ready to prove the main result in this paper.

\begin{theorem}\label{th:1}
Let $D$ be a division ring with infinite center $F$ and $K$ be a division subring of $D$. Assume that $N$ is a non-central almost subnormal subgroup of $D^*$. If $K$ is $N$-invariant, then $K\subseteq F$ or $K=D$. 
\end{theorem}

\begin{proof}
Assume that $N$ is an almost $r$-subnormal subgroup in $D^*$ with an almost normal series
$$N=N_r\leq N_{r-1}\leq \ldots \leq N_0=D^*.$$
Observe that if $N_r$ has finite index in $N_{r-1}$, then the core $$\Core_{N_{r-1}}(N_r):=\bigcap_{x\in N_{r-1}} N_r^x$$ 
of $N_r$ in $N_{r-1}$ is a normal subgroup of finite index in $N_{r-1}$. 

We claim that $\Core_{N_{r-1}}(N_r)$ is non-central. Indeed, if $\Core_{N_{r-1}}(N_r)$ is central, then $a^n\in \Core_{N_{r-1}}(N_r)$ for any $a\in N_{r-1}$, where $n=[N_{r-1}:\Core_{N_{r-1}}(N_r)]$. Hence, $a^nb^na^{-n}b^{-n}=1$ for very $a,b\in N_{r-1}$. Moreover, $x^ny^nx^{-n}y^{-n}$ is non-identity, so $N_{r-1}$ satisfies the identity $x^ny^nx^{-n}y^{-n}=1$. By \cite[Theorem 2.2]{Pa_NgBiHa_2016}, $N_{r-1}$ is central, so is $N=N_r\subseteq F$, a contradiction.
To prove the theorem, it suffices to use the fact that $K$ is normalized by the core of $N_r$ in $N_{r-1}$. So, without loss of generality, we can assume that $N_r$ is normal in $N_{r-1}$. We shall prove the theorem firstly for the case when $K$ is centrally finite, that is, when $K$ is a finite dimensional vector space over its center $Z(K)$, and then for the general case.

\bigskip
\noindent
{\it Case 1. $K$ is centrally finite.}
 \bigskip

We shall prove the statement by induction on $r$. If $r=0$, then $K^*$ is normal in $D^*$, and the statement is true by the Carter-Brauer-Hua Theorem. Assume that the statement holds for any almost subnormal subgroup in $D^*$ of depth $<r$. Assuming $K\not\subseteq F$, we must show $K=D$. In view of Lemma~\ref{lem:7}, $K$ is not commutative. We claim that $Z(K)$ is contained in $F$. Indeed, for any $x\in Z(K)$, $g\in N_r$ and $h\in K$, since $gKg^{-1}=K$, there exists $h'\in K$ such that $h=gh'g^{-1}$. One has $$(gxg^{-1})h=gxg^{-1}gh'g^{-1}=gxh'g^{-1}=gh'xg^{-1}=gh'g^{-1} gxg^{-1}=h(gxg^{-1}).$$ Hence, $gxg^{-1}\in Z(K)$ for any $x\in Z(K)$ and $g\in N_r$. This means that $Z(K)$ is $N_r$-invariant. By Lemma~\ref{lem:7}, $Z(K)\subseteq F$. The claim is proved. Since $K\not\subseteq F$, by Lemma~\ref{lem:6}, there exists an element $h\in (K\cap N)\backslash F$. Assume that $N_r\not\subseteq K$. Then, there exists $g\in N_r\backslash K$. Consider the morphism $\psi : K\to K$ defined by $\psi(x)=gxg^{-1}$ for any $x\in K$. This morphism is a $Z(K)$-automorphism of $K$ in view of the inclusion $Z(K)\subseteq F$. Since $K$ is finite dimensional over $Z(K)$, by the Skolem-Noether Theorem, there exists $a\in K^*$ such that $\psi(x)=axa^{-1}$ for any $x\in K$. In particular, $ghg^{-1}=aha^{-1}$, equivalently, $(a^{-1}g)h=h(a^{-1}g)$. Hence, $a^{-1}g\in C_D(h)$, but in view of Lemma~\ref{lem:9}, $C_D(h)\subseteq K$, so we have $g\in K$, a contradiction. Thus, $N_r\subseteq K$. As a corollary, the division subring $L$ of $D$ generated by $N_r$ is contained in $K$ and $L$ is $N_{r-1}$-invariant. By the inductive, $L=D$, and consequently, $K=D$.

\bigskip
\noindent{\it Case 2. General case.}
\bigskip

As in {\it Case 1}, we shall prove the statement by induction on $r$. If $r=0$, then the statement is true by the Carter-Brauer-Hua Theorem. Assume that the statement holds for any almost subnormal subgroup of $D$ of depth $<r$. Assuming $K\not\subseteq F$, we must show $K=D$. We claim that $N_r\subseteq K$. Assume that $N_r\not \subseteq K$. Then, take $g\in N_r\backslash K$, and by Lemma~\ref{lem:6}, let $h\in (K\cap N_r)\backslash F$. Consider $u=u_r(h,g)=(1+g)\phi_n(h,g) (1+g)^{-1}\in N_r$ and $v=v_r(h,g)=(1+g)^{-1}\phi_n(h,g) (1+g)\in N_r$ as in Lemma~\ref{lem:5}. One has $$\alpha(h,g)=[\phi_n(h,g),u]=[(1+g)v(1+g)^{-1}, (1+g)\phi_n(h,g)(1+g)^{-1}]$$$$=(1+g)[v,\phi_n(h,g)](1+g)^{-1}\in (1+g)K(1+g)^{-1}.$$ On the other hand $\alpha(h,g)=[\phi_n(h,g), u]\in [K\cap N_r, N_r]\subseteq K\cap N_r.$ Hence, $\alpha(h,g) \in N_r\cap K\cap (1+g)K(1+g)^{-1}$. In view of Corollary~\ref{cor:10}, $\alpha(h,g)\in F$. Observe that for any $g'\in K\cap N_r$, $gg'\in N_r\backslash K$, so $\alpha(h,gg')\in F\cap K$. Let us fix some element $e\in (K\cap N_r)\backslash F$, and recall elements $w_n(a,b)$ which have been constructed before Lemma~\ref{lem:5} for $a,b\in D^*$ and $n\geq 0$. We claim that $\alpha (h,gw_{r+1}(d,e))\in F\cap K\subseteq Z(K)$ for any $d\in D^*$. To do this, it suffices to show that $w_{r+1}(d,e)\in N_r\cap K$ for any $d\in D^*$. Recall that $e\in N_r\cap K$, so by a remark before Lemma~\ref{lem:5}, we have $w_r(d,e)\in N_r$ for any $d\in D^*$. Moreover, since $K$ is $N_r$-invariant, we also have 
$$w_{r+1}(d,e)=[w_r(d,e),e]=w_r(d,e)e(w_r(d,e))^{-1}e^{-1}\in K\cap N_r.$$ 
Thus, we have shown that $\alpha(h,gw_r(p,e))\in F$ for every $c\in D^*$. Hence, $$\alpha(h,gw_r(p,e))q-q\alpha(h,gw_r(p,e))=0$$ for every $p,q\in D^*$. Since $\alpha(h,gw_r(x,e))y-y\alpha(h,gw_r(x,e))$ is non-zero in $D(x,y)$, $D^*$ satisfies generalized rational identity $\alpha(h,gw_r(x,e))y-y\alpha(h,gw_r(x,e))=0$.
Therefore, by Lemma~\ref{hdb}, $D$ is centrally finite, so in view of \cite[Theorem 3]{hai-ngoc}, $K$ is centrally finite. By Case 1, $K=D$. But this fact contradicts the assumption that $N_r\not\subseteq K$. Thus, $N_r\subseteq K$, and the claim is proved. Now, as in Case 1, we conclude that the division subring $L$ of $D$ generated by $N_r$ is $N_{r-1}$-invariant, so $L=D$ by the inductive, and this implies $K=D$. The proof of the theorem is now complete.
\end{proof}


\bigskip

\noindent

\begin{thebibliography}{}
\bibitem{Bo_Cohn} P. M. Cohn, {\it Free rings and their relations}, Academic Press, New York and London, 1971.
\bibitem{Pa_GoMi_82} I. Z. Golubchik and A. V. Mikhalev, Generalized group identities in the classical groups, \textit{Zap. Nauch. Semin. LOMI AN SSSR} \textbf{114} (1982): 96--119.
\bibitem{Pa_Gr_60} L. Greenberg, Discrete groups of motions, \textit{Canad. J. Math.} \textbf{12} (1960), 414--425.
\bibitem{Pa_HaBiDu_16} B. X. Hai, M. H. Bien and T. H. Dung, Generalized algebraic rational identities of subnormal subgroups in division rings, \textit{arXiv:1709.04774v1[math RA]13 Sep. 2017}.
\bibitem{hai-ngoc} B. X. Hai and N. K. Ngoc, A note on the existence of non-cyclic free subgroups in division rings, \textit{Archiv der Math.} {\bf 101} (2013), 437--443.
\bibitem{Pa_Ha_89} B. Hartley, Free groups in normal subgroups of unit groups and arithmetic groups, \textit{Contemp. Math.} \textbf{93} (1989) 173--177.
\bibitem{Haz-Wad} R. Hazrat and A.R. Wadsworth, On maximal subgroups of the multiplicative group of a division algebra, {\em J. Algebra} {\bf 322} (2009), 2528--2543.
\bibitem{Pa_HeSc_63} I. N. Herstein and W. R. Scott, Subnormal subgroups of division rings, \textit{Canad. J. Math.} \textbf{15} (1963) 80--83.
\bibitem{Pa_KaSo_69} A. Karrass and D. Solitar, On Finitely Generated Subgroups of a Free Group, \textit{Proc. Amer. Math. Soc.} \textbf{22} (1969), 209--213.
\bibitem{lam} T. Y. Lam, {\em A first course in noncommutative Rings}, GMT {\bf 131}, Springer, 1991.
\bibitem{Bo_MaKaSo_76} W. Magnus, A. Karrass and D. Solitar, \textit{Combinatorial group theory. Presentation of groups in terms of generators and relations}. 2nd revised edition, Dover Publications, New York, 1976.
\bibitem{Pa_NgBiHa_2016} N. K. Ngoc, M. H. Bien and B. X. Hai, Free subgroups in almost subnormal subgroups of general skew linear groups, \textit{St. Petersburg Math. J.} \textbf{28} (2017), 707--717.
\bibitem{Pa_Ol_17} A. Y. Olshanskii, {Subnormal subgroups in free groups, their growth and cogrowth}, \textit{Math. Proc. Camb. Phil. Soc.} \textbf{169} (2017), 499--531.
\bibitem{scott} W. R. Scott, \textit{Group theory}, Dover Publications Inc., New York, second edition, 1987.
\bibitem{Bo_St_93} J. Stillwell, \textit{Classical Topology and Combinatorial Group Theory}, In: Graduate Texts in Mathematics, Vol \textbf{72}, 2nd edition, Springer-Verlag, 1993.
\bibitem{Stuth} C. J. Stuth, {A generalization of the Cartan-Brauer-Hua Theorem}, \textit{Proc. Amer. Math. Soc.} \textbf{15} (1964), 211--217.
\bibitem{Pa_To_85} G. M. Tomanov, Generalized group identities in linear groups, \textit{Math. USSR, Sbornik} \textbf {51} (1985), 33--46.
\bibitem{Wehr_93} B.A. F. Wehrfritz, A note on almost subnormal subgroups of linear groups, \textit{Proc. Amer. Math. Soc.} \textbf{117} (1993), no. 1, 17--21.
\end{thebibliography}
\end{document}